\documentclass{amsart}

\usepackage[english]{babel}
\usepackage{amsthm, latexsym, gastex, amssymb,url,listings}
\usepackage[mathcal]{eucal}
\usepackage{graphicx}

\DeclareMathOperator{\Tr}{Tr}

\DeclareMathOperator{\Kob}{Kob_0}

\renewcommand{\phi}[0]{\varphi}
\renewcommand{\theta}[0]{\vartheta}
\renewcommand{\epsilon}[0]{\varepsilon}

\newcommand{\N}{\text{$\mathbf{N}$}}

\newcommand{\Pro}{\text{$\mathbf{P}^1$}}
\newcommand{\F}{\text{$\mathbb{F}$}}

\newtheorem{theorem}{Theorem}[section]
\newtheorem{lemma}[theorem]{Lemma}
\newtheorem{corollary}[theorem]{Corollary}

\theoremstyle{definition}

\newtheorem{example}[theorem]{Example}

\theoremstyle{remark}
\newtheorem{remark}[theorem]{Remark}

\numberwithin{equation}{section}

\begin{document}
\title[Some notes on the multiplicative order of $\alpha + \alpha^{-1}$]{Some notes on the multiplicative order of $\alpha + \alpha^{-1}$ in finite fields of characteristic two}

\author{S.~Ugolini}
\copyrightinfo{2018}{Simone Ugolini}
\email{s.ugolini@unitn.it} 



\begin{abstract}
In this paper we prove some results on the possible multiplicative orders of  $\alpha + \alpha^{-1}$ when $\alpha$ is a non-zero element of a finite field of characteristic 2. The results of the paper rely on a previous investigation on the structure of the graphs associated with the map $x \mapsto x + x^{-1}$ in finite fields of characteristic 2.
\end{abstract}

\maketitle
\bibliographystyle{amsplain}

\maketitle

\section{Introduction}
Elements of the form $\alpha + \alpha^{-1}$, where $\alpha$ belongs to the multiplicative group $\F_q^*$ of a finite field $\F_q$ with $q$ elements, play a crucial role in many contexts.

For example, the following nice property holds for the Dickson polynomial $D_n(x)$ of the first kind with parameter $1$ and degree $n$ (see \cite{LN} or the monograph \cite{LMT}): 
\begin{displaymath}
D_n (x+x^{-1}) = x^n + x^{-n}.
\end{displaymath}

Meyn \cite{mey} and Varshamov and Garakov \cite{var} employed the $Q$-transform, which takes a polynomial $f(x)$ to $f^Q (x) = x^{\deg(f)} \cdot f(x +x ^{-1})$, for the recursive synthesis of irreducible polynomials.

Shparlinski \cite{shpa} answered a question posed in \cite[Research Problem 3.1]{BGM} about the possibility of finding the multiplicative order $|\gamma + \gamma^{-1}|$ of $\gamma + \gamma^{-1}$ from the knowledge of the multiplicative order $|\gamma|$ of $\gamma \in \F_q^*$. He showed that the orders of $\gamma$ and $\gamma + \gamma^{-1}$ are independent in a certain sense, even though in finite fields of small characteristic it is not possible that  $|\gamma|$ and $|\gamma + \gamma^{-1}|$ are both small (see \cite[Section 4]{vGS}).

In this paper we  consider finite fields of characteristic $2$. If $\F_q$ is a finite field of characteristic $2$, then we define the map $\theta$  over the projective line $\Pro (\F_q) := \F_q \cup \{ \infty \}$ as follows:
\begin{displaymath}
\theta(x) = 
\begin{cases}
\infty & \text{if $x \in \{0, \infty \}$;}\\
x + x^{-1} & \text{otherwise.}
\end{cases}
\end{displaymath}

Such a map is strictly related to the duplication map over a Koblitz curve (see \cite{SU2}). In Section \ref{preliminaries} we review some properties of the graph $G_q$ associated with the map $\theta$ over $\Pro(\F_q)$. The reader can refer to \cite{SU2} for the proofs of the results. Very briefly, we recall that the vertices of $G_q$ are the elements of $\Pro (\F_q)$. Moreover, for any $\alpha \in \Pro(\F_q)$ there is an arrow which joins $\alpha$ to $\theta (\alpha)$. We say that $\alpha \in \Pro(\F_q)$ is $\theta$-periodic if $\theta^m (\alpha) = \alpha$ for some positive integer $m$ (here $\theta^m$ denotes the $m$-fold composition of $\theta$ with itself). If $\alpha$ is not $\alpha$-periodic, then it is pre-periodic, namely some iterate of $\alpha$ is $\theta$-periodic.  The resulting graph is formed by a finite number of connected components. Each component is formed by a cycle, whose vertices are roots of binary trees.  

Relying upon some preliminary results presented in Section \ref{preliminaries}, in Section \ref{distribution} we prove some restrictions on the multiplicative order of the iterates $\theta^i (\gamma)$ of an element $\gamma$ in the multiplicative group of a finite field of characteristic $2$. The main result of the section is Theorem \ref{cen_thm}. In Section   \ref{Sub_distributions} we describe three possible scenarios for the order and the absolute trace of the iterates $\theta^i (\gamma)$ of an element $\gamma \in \F_{q^4} \backslash \{0, 1 \}$ whose multiplicative order divides $q^2+1$. 
Finally, in Section \ref{Sub-dickson} we show how the results of the current paper and \cite{SU2} can be related to some results on the roots of certain Dickson polynomials presented in \cite{Blo}.

\section{Notation}
For the reader's convenience we list here some notations we use along the way.
\begin{itemize}
\item If $G$ is a (multiplicative) group and $g$ is an element of $G$, then $|g|$ is the order of $g$ in $G$, namely $|g|$ is the smallest positive integer $n$ such that $g^n = 1$.
\item We denote by $\N$ the set of natural numbers ($0$ included). Moreover we define $\N^*:= \N \backslash \{ 0 \}$ and $\N^{**} := \N \backslash \{0 ,1 \}$.
\item If $L$ is a finite field, then we denote by $L^* := L \backslash \{ 0 \}$ its multiplicative group and we define $L^{**} := L \backslash \{ 0, 1 \}$.
\item If $L_1$ and $L_2$ are two finite fields with $L_1 \subseteq L_2$, then we denote by  $[L_2 : L_1]$ the degree of the extension $L_2$ over $L_1$.
\item If $\alpha$ is an element of a finite field $L$ of characteristic $2$, then we denote by $\deg(\alpha)$ the degree of the field extension $\F_2 (\alpha)$ over $\F_2$. Equivalently, $\deg(\alpha)$ is the degree of the minimal polynomial of $\alpha$ over $\F_2$.
\item If $\alpha \in \F_{2^t}$ for some positive integer $t$, then 
\begin{displaymath}
\Tr_{t} (\alpha) := \sum_{i=0}^{t-1} {\alpha}^{2^i}
\end{displaymath}
is the absolute trace of $\alpha$. We recall that $\Tr_{t} (\alpha) \in \F_2$.
\item If $f$ is a function from a set $A$ to a set $B$, then $f(A) := \{f(a) : a \in A \}$ is the image set of $A$ in $B$. 
\end{itemize}
\section{Preliminaries}
\label{preliminaries}
Throughout this section $t$ denotes a positive integer, which can be written as 
\begin{equation*}
t = 2^r \cdot s
\end{equation*}
for some non-negative integer $r$ and some odd integer $s$.

We define $q:=2^t$. 

We denote by $A_t$ and $B_t$ the following subsets of $\Pro(\F_q)$: 
\begin{align*}
A_t & :=  \{ \alpha \in \F_{q}^* : \Tr_{t} (\alpha) = \Tr_{t} (\alpha^{-1}) \} \cup \{ 0, \infty \}; \\
B_t & :=  \{ \alpha \in \F_{q}^* : \Tr_{t} ( \alpha ) \not = \Tr_{t} (\alpha^{-1})  \}.
\end{align*} 

The set $A_t$ is strictly related to the ring $E(\F_q)$ of rational points in $\F_q$ of the elliptic curve $E$ defined over $\F_2$ by the equation
\begin{equation*}
y^2+xy=x^3+1.
\end{equation*}

In fact, the following holds (see \cite[Lemma 2.5]{SU2}).

\begin{lemma}\label{pre_1}
Let $x \in \F_{q}$. Then there exists $y \in \F_{q}$ such that $(x,y) \in E (\F_{q})$ if and only if $x = 0$ or $x \not = 0$ and $\Tr_{t}(x) = \Tr_{t}(x^{-1})$.
\end{lemma}
\begin{remark}
We notice that in \cite{SU2} the curve $E$ was denoted by $\Kob$ since $E$ is a Koblitz curve. In the present paper we opted for the shorter notation $E$.
\end{remark}

In \cite{SU2} we described the structure of the directed graph $G_q$ associated with the map $\theta$ over $\Pro(\F_q)$. The following theorem summarizes the results contained in \cite[Remark 2.3, Lemma 4.3, Lemma 4.4]{SU2}.

\begin{theorem}\label{thm_gq}
Let $d:=r+2$. Then the following hold.
\begin{enumerate}
\item The image set $\theta(A_t)$ is contained in $A_t$, while $\theta(B_t)$ is contained in $B_t$. In particular, either all the vertices of a connected component are contained in $A_t$, or all of them are contained in $B_t$.  
\item If $x \in A_t$ is $\theta$-periodic, then $x$ is the root of a binary tree of depth $d$.
\item If $x \in B_t$ is $\theta$-periodic, then $x$ is the root of a binary tree of depth $1$.
\item If $x \in A_t \backslash \{ \infty \}$ is $\theta$-periodic, then for any positive integer $k \leq d$ there are  $\lceil 2^{k-1} \rceil$ vertices at the level $k$ of the tree. Moreover, the root has one child, while all other vertices have two children. 
\item $\infty$ is $\theta$-periodic and for any positive $k \leq d$ there are  $\lceil 2^{k-2} \rceil$ vertices at the level $k$ of the tree. Moreover, $\infty$ and the vertex at the level $1$ have one child, while all other vertices have two children.
\end{enumerate}
\end{theorem} 

\begin{example}
In this example we construct the graph $G_{2^6}$. The labels are the exponents of the powers $\alpha^i$, where $\alpha$ is a root of the Conway polynomial $x^6+x^4+x^3+x+1 \in \F_2 [x]$, with $0 \leq i \leq 62$. Moreover there is one vertex labelled by $\infty$ and another by `0', namely the zero of $\F_{2^6}$.

We notice in passing that the trees of any connected component have the same depth, which is either 1 or 3.

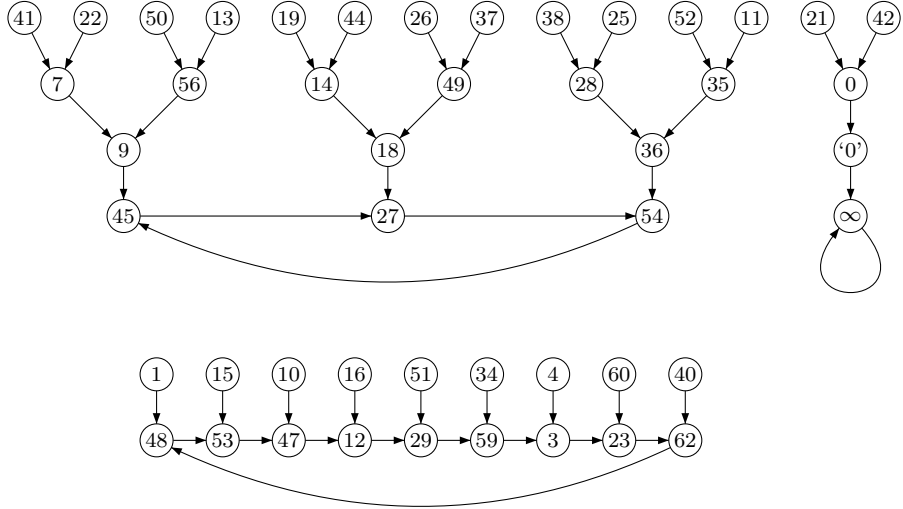
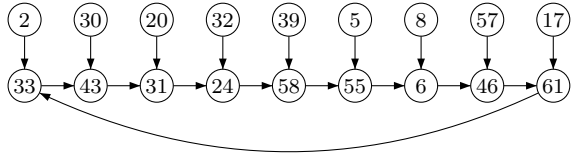
\begin{center}
    \unitlength=2.5pt
    \begin{picture}(130, 50)(0,-10)
    \gasset{Nw=5,Nh=5,Nmr=2.5,curvedepth=0}
    \thinlines
    \footnotesize
  
    \node(A1)(0,30){41}
    \node(A2)(10,30){22}
    \node(A3)(20,30){50}
    \node(A4)(30,30){13}
    \node(A5)(40,30){19}
    \node(A6)(50,30){44}
    \node(A7)(60,30){26}
    \node(A8)(70,30){37}
    \node(A9)(80,30){38}
    \node(A10)(90,30){25}
    \node(A11)(100,30){52}
    \node(A12)(110,30){11}
    
    \node(A13)(120,30){21}
    \node(A14)(130,30){42}

    \node(B1)(5,20){7}
    \node(B2)(25,20){56}
    \node(B3)(45,20){14}
    \node(B4)(65,20){49}
    \node(B5)(85,20){28}
    \node(B6)(105,20){35}
    
    \node(B7)(125,20){0}
        
    \node(C1)(15,10){9}
    \node(C2)(55,10){18}
    \node(C3)(95,10){36}
    
    \node(C4)(125,10){`0'}
    
    \node(D1)(15,0){45}
    \node(D2)(55,0){27}
    \node(D3)(95,0){54}
    
    \node(D4)(125,0){$\infty$}

    \drawedge(A1,B1){}
    \drawedge(A2,B1){}
    \drawedge(A3,B2){}
    \drawedge(A4,B2){}
    \drawedge(A5,B3){}
    \drawedge(A6,B3){}
    \drawedge(A7,B4){}
    \drawedge(A8,B4){}
    \drawedge(A9,B5){}
    \drawedge(A10,B5){}
    \drawedge(A11,B6){}
    \drawedge(A12,B6){}
    \drawedge(A13,B7){}
    \drawedge(A14,B7){}
    \drawedge(B1,C1){}
    \drawedge(B2,C1){}
    \drawedge(B3,C2){}
    \drawedge(B4,C2){}
    \drawedge(B5,C3){}
    \drawedge(B6,C3){}
    \drawedge(B7,C4){}
    \drawedge(C1,D1){}
    \drawedge(C2,D2){}
    \drawedge(C3,D3){}
    \drawedge(C4,D4){}
    \drawedge(D1,D2){}
    \drawedge(D2,D3){}
    
    \drawloop[loopangle=-90](D4){}
    
    \gasset{curvedepth=10}
    \drawedge(D3,D1){}
\end{picture}
\end{center}

\begin{center}
    \unitlength=2.5pt
    \begin{picture}(90, 30)(0,-10)
    \gasset{Nw=5,Nh=5,Nmr=2.5,curvedepth=0}
    \thinlines
    \footnotesize
  
    \node(A1)(0,10){1}
    \node(A2)(10,10){15}
    \node(A3)(20,10){10}
    \node(A4)(30,10){16}
    \node(A5)(40,10){51}
    \node(A6)(50,10){34}
    \node(A7)(60,10){4}
    \node(A8)(70,10){60}
    \node(A9)(80,10){40} 
    
    \node(B1)(0,0){48}
    \node(B2)(10,0){53}
    \node(B3)(20,0){47}
    \node(B4)(30,0){12}
    \node(B5)(40,0){29}
    \node(B6)(50,0){59}
    \node(B7)(60,0){3}
    \node(B8)(70,0){23}
    \node(B9)(80,0){62}

    \drawedge(A1,B1){}
    \drawedge(A2,B2){}
    \drawedge(A3,B3){}
    \drawedge(A4,B4){}
    \drawedge(A5,B5){}
    \drawedge(A6,B6){}
    \drawedge(A7,B7){}
    \drawedge(A8,B8){}
    \drawedge(A9,B9){}
    
    \drawedge(B1,B2){}
    \drawedge(B2,B3){}
    \drawedge(B3,B4){}
    \drawedge(B4,B5){}
    \drawedge(B5,B6){}
    \drawedge(B6,B7){}
    \drawedge(B7,B8){}
    \drawedge(B8,B9){}
    
    \gasset{curvedepth=10}
    
    \drawedge(B9,B1){}
\end{picture}
\end{center}

\begin{center}
    \unitlength=2.5pt
    \begin{picture}(90, 30)(0,-10)
    \gasset{Nw=5,Nh=5,Nmr=2.5,curvedepth=0}
    \thinlines
    \footnotesize
  
    \node(A1)(0,10){2}
    \node(A2)(10,10){30}
    \node(A3)(20,10){20}
    \node(A4)(30,10){32}
    \node(A5)(40,10){39}
    \node(A6)(50,10){5}
    \node(A7)(60,10){8}
    \node(A8)(70,10){57}
    \node(A9)(80,10){17} 
    
    \node(B1)(0,0){33}
    \node(B2)(10,0){43}
    \node(B3)(20,0){31}
    \node(B4)(30,0){24}
    \node(B5)(40,0){58}
    \node(B6)(50,0){55}
    \node(B7)(60,0){6}
    \node(B8)(70,0){46}
    \node(B9)(80,0){61}

    \drawedge(A1,B1){}
    \drawedge(A2,B2){}
    \drawedge(A3,B3){}
    \drawedge(A4,B4){}
    \drawedge(A5,B5){}
    \drawedge(A6,B6){}
    \drawedge(A7,B7){}
    \drawedge(A8,B8){}
    \drawedge(A9,B9){}
    
    \drawedge(B1,B2){}
    \drawedge(B2,B3){}
    \drawedge(B3,B4){}
    \drawedge(B4,B5){}
    \drawedge(B5,B6){}
    \drawedge(B6,B7){}
    \drawedge(B7,B8){}
    \drawedge(B8,B9){}
    
    \gasset{curvedepth=10}
    
    \drawedge(B9,B1){}
\end{picture}
\end{center}
\end{example}

We recall some well-known facts we need in the rest of the paper. 

\begin{lemma}\label{elm_1}
The following hold:
\begin{itemize}
\item $\gcd(2^t-1, 2^t+1)= 1$;
\item $\gcd(2^t+1, 2^{2t}+1) =1$.
\end{itemize}
\end{lemma}
\begin{proof}
If $d = \gcd(2^t-1, 2^t+1)$, then $d \mid ((2^t+1)-(2^t-1)) = 2$. Since $d$ cannot be equal to $2$, we conclude that $d=1$.

If $e = \gcd(2^t+1, 2^{2t}+1)$, then $e \mid ((2^{2t}+1)-(2^t+1)) = 2^t (2^t-1)$. Since $e$ cannot be even, we have that $e \mid (2^t-1)$. Hence $e \mid \gcd(2^t-1, 2^t+1) = 1$. Therefore $e=1$.
\end{proof}

\begin{lemma}\label{elm_4}
Let $g_1$ and $g_2$ be two elements of a finite commutative group $G$. If $\gcd(|g_1|, |g_2|) = 1$, then $|g_1 g_2| = |g_1| \cdot |g_2|$.
\end{lemma}

\begin{lemma}\label{elm_2}
The group $\F_q^*$ is cyclic of order $2^t-1$. In particular, if $\alpha \in \F_q^{ ** }$ then $|\alpha| \nmid (2^t+1)$.
\end{lemma}
\begin{proof}
For the first part of the claim see for example \cite[Theorem 2.8]{LN}.
For the second part of the claim, we notice that $|\alpha| > 1$. Since $|\alpha| \mid (2^t-1)$, according to Lemma \ref{elm_1} we have that $|\alpha| \nmid (2^t+1)$. 
\end{proof}

We recall that $q + 1 = (2^s)^{2^r}+1$ is a generalized Fermat number. The following holds for its divisors.

\begin{lemma}\label{elm_5}
Let $d$ be a positive divisor of $q + 1$. Then $d \equiv 1 \bmod{2^{r+1}}$.
\end{lemma}
\begin{proof}
Let $p$ be a prime which divides $(2^s)^{2^{r}} + 1$. According to \cite{DK}, $p = k \cdot 2^{r+1} + 1$ for some $k \in \N^*$, namely 
\begin{equation}\label{p_cong}
p \equiv 1 \bmod{2^{r+1}}.
\end{equation}
Since $q + 1$ can be written as a product of primes satisfying (\ref{p_cong}), we get the result.
\end{proof}

\begin{lemma}\label{elm_3}
If $\alpha$ has degree $t$ over $\F_2$ and $\beta := \alpha + \alpha^{-1}$, then either $\deg(\beta) = t$ or $\deg(\beta) = \frac{t}{2}$ (provided that $t$ is even).
\end{lemma}
\begin{proof}
Let $L := \F_2(\beta)$. Then $\alpha$ is a root of the polynomial $p(x) := x^2 + \beta x + 1 \in L [x]$. The result follows because 
\begin{equation*}
\deg(\beta) = [L:\F_2] = \frac{t}{[\F_2(\alpha):L]}   
\end{equation*}
with $1 \leq [\F_2(\alpha):L] \leq 2$.
\end{proof}

The following lemma (see \cite[Lemma 4.1]{LW}) and Lemma \ref{lem_lac} will be used repeatedly in the paper. 
\begin{lemma}\label{lem_lac_0}
Let $\alpha$ be an element of order $2^t-1$ in $\F_{2^t}^*$ and $\beta$ an element of order $2^t+1$ in $\F_{2^{2t}}^*$.
Let $\Omega$ and $\overline{\Omega}$ be two subsets of $\F_{2^t}$ defined as 
\begin{align*}
\Omega & = \{x \in \F_{2^t}^* : \Tr_t (x^{-1}) = 0 \};\\
\overline{\Omega} & = \{x \in \F_{2^t}^* : \Tr_t (x^{-1}) = 1 \}.
\end{align*}
Then
\begin{align*}
\Omega & = \{x \in \F_{2^t}^* : x = \theta(\alpha^i) \text{ with $i \in \N$, $1 \leq i \leq 2^{t-1}-1$ } \};\\
\overline{\Omega} & = \{x \in \F_{2^t}^* :  x = \theta(\beta^i) \text{ with $i \in \N$, $1 \leq i \leq 2^{t-1}$ }  \}.
\end{align*} 
\end{lemma}
\begin{remark}
Let 
\begin{align*}
I & := \{i \in \N : 1 \leq i \leq 2^{t-1} \};\\
J & := \{j \in \N: 2^{t-1} + 1 \leq j \leq 2^t \}.
\end{align*}
If $\beta$ is defined as in Lemma \ref{lem_lac_0}, then $\beta^k + \beta^{-k} \in \F_{2^t}^*$ for any $k \in I \cup J$. 

Indeed, the assertion is true if $k \in I$, according to Lemma \ref{lem_lac}. 

Let $k \in J$. We notice that
\begin{displaymath}
\beta^{-k} = \beta^{2^t+1-k}.
\end{displaymath}

Let $i := 2^t+1-k$. Then 
\begin{displaymath}
1 \leq i \leq 2^{t-1}.
\end{displaymath}
Hence $\beta^k + \beta^{-k} = \beta^i + \beta^{-i} \in \overline{\Omega} \subseteq \F_{2^t}^*$.
\end{remark}

\begin{remark}
We notice that $\beta^{2t+1} = 1$. Hence $\beta^{2t+1} + \beta^{-(2t+1)} = 0 \not \in \F_{2^t}^*$. 
\end{remark}

From Lemma \ref{lem_lac_0} and the subsequent remarks we deduce the following results, where $\overline{\Omega}$ is defined as in Lemma \ref{lem_lac_0}.
\begin{lemma}\label{lem_lac}
We have that $\overline{\Omega} = \{\theta(\gamma): \gamma \in \F_{2^{2t}}^{**}, |\gamma| \mid  (2^t+1) \}$. 
\end{lemma}

\begin{corollary}\label{cor_lac}
Let $x \in \F_q^*$ be a leaf of a connected component of $G_q$.
\begin{itemize}
\item If $x \in A_t$, then $\Tr_t(x) = \Tr_t(x^{-1}) = 1$.
\item If $x \in B_t$, then $\Tr_t(x) = 0$ and $\Tr_t(x^{-1}) = 1$. 
\end{itemize}
\end{corollary}
\begin{proof}
We notice that $x \in \Omega \cup \overline{\Omega}$. Since $x$ is a leaf, there is no $\tilde{x} \in \F_q^*$ such that $x = \theta (\tilde{x})$. Therefore $x \in\overline{\Omega}$.
\end{proof}

The following property holds for the leaves belonging to $A_t$.

\begin{lemma}\label{iter_trace}
Let $\alpha \in \F_q^* \cap A_t$ be a leaf of $G_q$ with $\Tr_t (\alpha) = \Tr_t (\alpha^{-1}) =1$. 

If $\theta(\gamma) = \alpha$ for some $\gamma \in \F_{q^2}^*$, then $\Tr_{2t} (\gamma) = \Tr_{2t} (\gamma^{-1}) = 1$.
\end{lemma}
\begin{proof}
First we notice that $q^2 = 2^{2t} = 2^{2^{r+1} s}$.
Therefore $\gamma$ belongs to the level $r+3$ of a tree of $G_{q^2}$. According to Theorem \ref{thm_gq} the trees in $G_{q^2}$ have depth either $1$ or $r+3$ and the former holds only for trees whose vertices belong to $B_{2t}$. Since $r+3 > 1$ we conclude that $\gamma$ is a leaf which belongs to $A_{2t}$. 
\end{proof}
  
\section{Distribution of the orders in $G_q$, $G_{q^2}$, $G_{q^4}$ and some implications}\label{distribution}
Let $n := 2^l \cdot m$ for some non-negative integer $l$ and some odd integer $m$.

Let $q:=2^n$.

First we prove some results on the degrees of $\alpha$ and its iterates $\theta^i (\alpha)$ for an element $\alpha$ belonging to a finite field of characteristic $2$.

\begin{lemma}\label{lem_leaf_0}
Let $\alpha$ be a leaf of $G_{2^t}$, where $t = 2^r \cdot s$ for some non-negative integer $r$ and some odd integer $s$. 

Then $\alpha \not \in \{0, \infty \}$ and $\deg(\alpha) = 2^r \cdot v$ for some odd integer $v$ such that $v \mid s$. 
\end{lemma}
\begin{proof}
First we notice that $\theta(0) = \infty$ and $\theta(1) = 0$. Hence $\alpha \in \F_{2^t}^*$.

Let $\deg(\alpha) = 2^u \cdot v$ for some non-negative integer $u$ and some odd integer $v$. Since $\deg(\alpha) \mid t$, we have that $u \leq r$ and $v \mid s$. Indeed, $u = r$. In fact, if  $u < r$ and $\gamma$ is an element in the algebraic closure $\overline{\F}_2$ of $\F_2$ such that $\theta(\gamma) = \alpha$, then $\gamma \in G_{2^t}$, in contradiction with the fact that $\alpha$ is a leaf of $G_{2^t}$.
\end{proof}

\begin{lemma}\label{lem_leaf_1}
Let $t := 2^r \cdot s$ and $\tilde{t} := 2^{r-1} \cdot s$ for some positive integer $r$ and some odd integer $s$.
Let $\alpha \in A_t \cap \F_{2^t}^*$ and $d:=\deg(\alpha)$.

If $\alpha$ is a leaf of $G_{2^t}$ and $\deg(\theta (\alpha)) = d$, then one of the following holds:
\begin{enumerate}
\item $\deg(\theta^i(\alpha)) = d$ for all $i \in \N$.
\item $\theta^i(\alpha) \in B_{\tilde{t}} \cap \F_{2^{\tilde{t}}}^*$ for any $i \geq r+1$, while $\deg(\theta^i(\alpha)) = d$ for $1 \leq i \leq r$.
\end{enumerate}
\end{lemma}
\begin{proof}
According to Lemma \ref{lem_leaf_0} we have that $d=2^{r} \cdot v$ for some odd integer $v \mid s$. 

Now we suppose that $\deg(\theta^i(\alpha)) \not = d$ for some $i \in \N$. We define 
\begin{equation*}
j := \min \{i \in \N : \deg(\theta^i(\alpha)) \not = d \}.
\end{equation*}
We notice that $\deg(\theta^j(\alpha)) = \frac{d}{2} = 2^{r-1} \cdot v$. Hence 
\begin{equation*}
\theta^j (\alpha) \in \F_{2^{\tilde{t}}} \text{ and } \theta^{j-1} (\alpha) \not \in \F_{2^{\tilde{t}}}.
\end{equation*}

Therefore $\theta^j (\alpha)$ is a leaf of the graph $G_{2^{\tilde{t}}}$, whose connected components have depth either $1$ or $r+1 \geq 2$. Hence either $j = r+1$ or $j = 1$. This latter is not possible because $\deg(\theta(\alpha)) = d$ by hypothesis. Therefore $j=r+1$.
\end{proof}

In the following theorem we investigate the possible orders of the iterates of an element in $\F_{q^4}^{**}$ whose multiplicative order divides $q^2+1$.

\begin{theorem}\label{cen_thm}
Let $C_{q^2+1}$ be the cyclic subgroup of $\F_{q^4}^*$ of order $q^2+1$ and $H := C_{q^2+1} \backslash \{ 1\}$. Then $H = H_1 \cup H_2 \cup H_3$, where $H_1, H_2$ and $H_3$ are three subsets of $H$, whose pairwise intersections are empty, characterized as follows.
\begin{enumerate}
\item $\gamma \in H$ belongs to $H_1$ if and only if $|\theta(\gamma)|$ divides $q+1$ and $|\theta^i (\gamma)|$ divides $q-1$ for all the integers $i \geq 2$. 

\item $\gamma \in H$ belongs to $H_2$ if and only if
\begin{itemize}
\item for any $i \in \N$ with $1 \leq i \leq l+1$ there are two positive integers $d_i$ and $e_i$ greater than $1$  with $d_i \mid (q+1)$ and $e_i \mid (q-1)$ such that $|\theta^i (\gamma)| = d_i e_i$;
\item $|\theta^{l+2} (\gamma)|$ divides $q+1$;
\item $|\theta^i (\gamma)|$ divides $q-1$ for all $i \geq l+3$.
\end{itemize}
\item $\gamma \in H$ belongs to $H_3$ if and only if for any positive integer $i$ there are two positive integers $d_i$ and $e_i$ greater than $1$ with $d_i \mid (q+1)$ and $e_i \mid (q-1)$ such that $|\theta^i (\gamma)| = d_i e_i$.
\end{enumerate}
\end{theorem}
\begin{proof}
First we notice that $q^2 = 2^{2n}$.

We prove that the pairwise intersections of the sets $H_1$, $H_2$ and $H_3$ are empty.
If $\gamma \in H_1$, then $|\theta(\gamma)|$ divides $q+1$. By definition, if $\gamma \in H_2 \cup H_3$, then $|\theta(\gamma)|$ does not divide $q+1$. Therefore $H_1 \cap H_2 = \emptyset$ and $H_1 \cap H_3 = \emptyset$. Finally, if $\gamma \in H_2$, then $|\theta^{l+3} (\gamma)|$ is a divisor of $q-1$. Therefore $\gamma \not \in H_3$.

Now we prove that, if $\gamma \in H$, then $h \in H_1 \cup H_2 \cup H_3$.
We deal separately with some cases.

\begin{itemize}
\item \emph{Case 1:} $|\theta(\gamma)|$ divides $q+1$.

Then $\theta^2(\gamma) \in \F_q^*$ according to Lemma \ref{lem_lac} and any iterate $\theta^i(\gamma)$ belongs to $\F_q^*$ for $i \geq 2$. Hence $\gamma \in H_1$.

\item \emph{Case 2:} $|\theta(\gamma)|$ does not divide $q+1$.

Let $t:=2n$. According to Lemma \ref{lem_lac_0} we have that $\theta(\gamma) \in \F_{q^2}^*$. Nevertheless, $\theta(\gamma) \not \in \F_q^*$. In fact, according to Lemma \ref{lem_lac}, any element in $\F_q^*$ can be expressed as 
\begin{displaymath}
\theta(\alpha^i) \ \text{ or } \ \theta(\beta^i)
\end{displaymath} 
for some positive integer $i$, where $|\alpha| = 2^n-1$ and $| \beta | = 2^n+1$. Since $|\gamma| \mid (2^{2n}+1)$ with $|\gamma| > 1$ and $\gcd(2^{2n}+1, 2^{2n} - 1) =1$, we conclude that $\gamma \not \in \F_q^*$.
Hence 
\begin{align*}
\gamma & \in \F_{q^4}^* \backslash \F_{q^2}^*,\\
\theta(\gamma) & \in \F_{q^2}^* \backslash \F_q^*.
\end{align*}

 Moreover $\theta(\gamma)$ is a leaf of $G_{q^2}$ because $\gamma \not \in \F_{q^2}$
 and either $\theta(\gamma) \in A_{2n}$ or $\theta(\gamma) \in B_{2n}$. 

\begin{itemize}
\item \emph{Sub-case 1:}  $\theta(\gamma) \in A_{2n}$. 

Let $\alpha := \theta(\gamma)$.
According to Lemma \ref{lem_leaf_0} we have that $d:= \deg(\alpha) = 2^{l+1} \cdot v$ for some odd integer $v \mid s$. 

Let $r:=l+1$, $s:=m$, $t:=2^r \cdot s$ and $\deg(\alpha) = d$.
According to Lemma \ref{elm_3} we have that $\deg (\theta(\alpha)) \in \left\{d, \frac{d}{2} \right\}$. If $\deg(\theta(\alpha)) = \frac{d}{2}$ then $\theta(\alpha) \in \F_{q}^*$ and $|\theta(\alpha)| \mid (q-1)$. This latter is not possible because $|\alpha| \nmid (q+1)$. Hence $\deg(\theta(\alpha)) = d$. 

Therefore the hypotheses of Lemma $\ref{lem_leaf_1}$ are satisfied.

If $\deg(\theta^i (\alpha)) = d$ for all $i \in \N$, then there is no $i \in \N$ such that $|\theta^i (\alpha)| \mid (q-1)$. In fact, if $|\theta^i (\alpha)| \mid (q-1)$ for some $i \in \N$, then $\theta^i (\alpha) \in \F_q$ and $\deg(\theta^i (\alpha)) < d$.

Moreover, if $|\theta^i (\alpha)| \mid (q+1)$ for some $i \in \N$, then $|\theta^{i+1} (\alpha)| \mid (q-1)$ according to Lemma \ref{lem_lac}, namely $\theta^{i+1} (\alpha) \in \F_q$. Since this latter is not possible, we conclude that $|\theta^i (\gamma)| = d_i e_i$ for some positive integers $d_i$ and $e_i$ greater than $1$ with $d_i \mid (q+1)$ and $e_i \mid (q-1)$. Therefore $\gamma \in H_3$.

If $\deg(\theta^i(\alpha)) \not = d$ for some $i \in \N$, then $\theta^i (\alpha) \in B_{t/2} \cap \F_q^*$ for any $i \geq r+1 = l+2$, namely  $\theta^i (\gamma) \in B_{t/2} \cap \F_q^*$ for any $i \geq l+3$. Moreover $\deg(\theta^i(\gamma)) = d$ for any $i$ with $1 \leq i \leq l+2$. Since $\theta^{l+3} (\gamma) \in \F_q^*$, according to Lemma \ref{lem_lac} we have that $|\theta^{l+2} (\gamma)|$ divides $q+1$. Finally, we can prove as above that $|\theta^{i} (\gamma)|$ is neither a divisor of $q+1$ nor a divisor of $q-1$ for all $i$ with $1 \leq i \leq l+1$. Hence $\gamma \in H_2$.

\item \emph{Sub-case 2:}  $\theta(\gamma) \in B_{2n}$. Then $\theta(\gamma)$ is a leaf of a tree of depth $1$ in $G_{q^2}$ and both the elements $\theta(\gamma)$ and $(\theta(\gamma))^{-1}$ do not belong to $\F_{q}^*$. Moreover $(\theta(\gamma))^{-1}$ is $\theta$-periodic and belongs to the same cycle of the iterates $\theta^i (\gamma)$ for any positive integer $i$.  Therefore $\theta^i (\gamma) \not \in \F_{q}^*$ for any positive integer $i$ and $|\theta^i (\gamma)|$ does not divide $q-1$. Moreover, if $|\theta^i (\gamma)|$ divided $q+1$ for some positive integer $i$, then $\theta^{i+1} (\gamma)$ would belong to $\F_q^*$ according to Lemma \ref{lem_lac}, in contradiction with the fact that no iterate of $\gamma$ belongs to $\F_q^*$. Since $|\theta^i (\gamma)|$ divides $q^2-1$, we conclude that for any positive integer $i$ there are two positive integers $d_i$ and $e_i$ greater than $1$ with $d_i \mid (q+1)$ and $e_i \mid (q-1)$ such that $|\theta^i (\gamma)| = d_i e_i$. Hence $\gamma \in H_3$. \qedhere
\end{itemize}
\end{itemize}
\end{proof}

\begin{corollary}\label{cen_cor}
If $H, H_1, H_2$ and $H_3$ are defined as in Theorem \ref{cen_thm}, then the following hold.
\begin{enumerate}
\item If $\gamma \in H_1$, then $\Tr_n (\theta^2 (\gamma)) = \Tr_n ((\theta^2 (\gamma))^{-1}) = 1$, while $\Tr_n (\theta^i (\gamma)) = \Tr_n ((\theta^i (\gamma))^{-1}) = 0$ for all $i \geq 2$.
\item If $\gamma \in H_2$, then $\Tr_n (\theta^{l+3} (\gamma)) = 0$ and $\Tr_n ((\theta^{l+3} (\gamma))^{-1}) = 1$. Moreover $\Tr_n (\theta^i (\gamma)) = 1$ and $\Tr_n ((\theta^i (\gamma))^{-1}) = 0$ for all $i \geq l+4$.
\end{enumerate}
\end{corollary}
\begin{proof}
We prove separately the claims.
\begin{enumerate}
\item If $\gamma \in H_1$ then, according to Theorem \ref{cen_thm},
\begin{align*}
\gamma & \in \F_{q^4}^* \backslash \F_{q^2}^*,\\
\theta(\gamma) & \in \F_{q^2}^* \backslash \F_q^*, \\
\theta^2(\gamma) & \in \F_q^*.
\end{align*}

We notice that $\theta^2 (\gamma)$ cannot be $\theta$-periodic. In fact, if $\theta^2(\gamma)$ were $\theta$-periodic, then $\theta^2 (\gamma)$ would be the root of a tree of depth at least 1 in $G_q$ and $\theta(\gamma)$, which is not $\theta$-periodic, would be a child of $\theta^2 (\gamma)$  in $G_q$. This latter is absurd because $\theta (\gamma) \in \F_{q^2}^* \backslash \F_q^*$.  

Since $\theta^2 (\gamma)$ is not $\theta$-periodic, $\theta(\gamma)$ is not $\theta$-periodic too. Therefore $\gamma$ belongs to a level not smaller than $4$ of a tree in $G_{q^4}$. Consequently $\theta^2 (\gamma)$ is a leaf of a tree having depth at least $2$ in $G_q$. Therefore $\theta^2(\gamma) \in A_n$.

\item If $\gamma \in H_2$, then $\theta(\gamma)$ is a leaf of a tree in $G_{q^2}$. Since $\theta(\gamma) \not \in \F_q$ and $\theta^{l+2} (\gamma) \not \in \F_q$, while $\theta^{l+3} (\gamma) \in \F_q$, we can say that neither $\theta(\gamma)$ nor $\theta^{l+2} (\gamma)$ are $\theta$-periodic. Therefore $\theta (\gamma)$ is a leaf of a tree of depth $l+3$ in $G_{q^2}$, while $\theta^{l+3} (\gamma)$ is a leaf of a tree of depth $1$ in $G_q$. Therefore $\theta^{l+3} (\gamma)$ and all its iterates are in $B_{n}$. \qedhere
\end{enumerate}
\end{proof}

\subsection{Distribution of orders and traces in $G_q$, $G_{q^2}$ and $G_{q^4}$}\label{Sub_distributions}
As a consequence of Theorem \ref{cen_thm} and Corollary \ref{cen_cor} we have three possible scenarios  for the iterates of an element $\gamma \in \F_{q^4}^{**}$ such that $d:=| \gamma |$ divides $(q^2+1)$. The data are collected in three distinct tables. 

For the sake of brevity, we introduce the following notations:
\begin{itemize}
\item $\gamma_i$ is the $i$-th iterate $\theta^i(\gamma)$ of $\gamma$;
\item $d_i := |\gamma_i|$;
\item $d_i \nmid (q \pm 1)$ stands for $d_i \nmid (q + 1)$ and $d_i \nmid (q-1)$.
\end{itemize}

Moreover we conventionally define $0^{-1} = 0$,  $\infty^{-1} = \infty$, $|0| = 1$, $|\infty|=1$ and $\Tr_n(\infty) = 0$.

The levels refer to the graph $G_{q^4}$.

\subsubsection{Case 1: $d_1 \mid (q+1)$}
\mbox{}

\begin{center}
\begin{tabular}{|c|c|c|}
\hline 
Level  & Order  & Trace \\ 
\hline 
$l+4$ & $d \mid (q^2+1)$ &  $\Tr_{4n} (\gamma) = \Tr_{4n} (\gamma^{-1}) = 1$ \\ 
\hline 
$l+3$ & $d_1 \mid (q+1)$ & $\Tr_{2n} (\gamma_1) = \Tr_{2n} (\gamma_1^{-1}) = 1$ \\ 
\hline 
$l+2$ & $d_2 \mid (q-1)$ & $\Tr_{n} (\gamma_2) = \Tr_{n} (\gamma_2^{-1}) = 1$ \\ 
\hline 
$l+1$ & \vdots & $\Tr_{n} (\gamma_3) = \Tr_{n} (\gamma_3^{-1}) = 0$ \\ 
\hline 
\vdots & \vdots & \vdots \\ 
\hline 
$0$ &  \vdots  & \vdots \\ 
\hline 
\end{tabular} 
\end{center}

\subsubsection{Case 2:} $d_1 \nmid (q+1)$, $d_{l+2} \mid (q+1)$.
\mbox{}

\begin{center}
\begin{tabular}{|c|c|c|}
\hline 
Level  & Order  & Trace \\ 
\hline 
$l+4$ & $d \mid (q^2+1)$ &  $\Tr_{4n} (\gamma) = \Tr_{4n} (\gamma^{-1}) = 1$ \\ 
\hline 
$l+3$ & $d_1 \mid (q^2-1), d_1 \nmid (q \pm 1)$ & $\Tr_{2n} (\gamma_1) = \Tr_{2n} (\gamma_1^{-1}) = 1$ \\ 
\hline 
$l+2$ & $d_2 \mid (q^2-1), d_2 \nmid (q \pm 1)$ & $\Tr_{2n} (\gamma_2) = \Tr_{2n} (\gamma_2^{-1}) = 0$ \\
\hline
\vdots & \vdots & \vdots \\ 
\hline 
$3$ & \vdots & \vdots \\
\hline 
$2$ & $d_{l+2} \mid (q+1)$ & $\Tr_{2n} (\gamma_{l+2}) = \Tr_{2n} (\gamma_{l+2}^{-1}) = 0$ \\ 
\hline 
$1$ & $d_{l+3} \mid (q-1)$ & $\Tr_{n} (\gamma_{l+3}) = 0, \Tr_{n} (\gamma_{l+3}^{-1}) = 1$  \\ 
\hline 
$0$ & $d_{l+4} \mid (q-1)$ & $\Tr_{n} (\gamma_{l+4}) = 1, \Tr_{n} (\gamma_{l+4}^{-1}) = 0$ \\ 
\hline 
\end{tabular} 
\end{center}

\subsubsection{Case 3:} $d_1 \nmid (q+1)$, $d_{l+2} \nmid (q+1)$.
\mbox{}

\begin{center}
\begin{tabular}{|c|c|c|}
\hline 
Level  & Order  & Trace \\ 
\hline 
$l+4$ & $d \mid (q^2+1)$ &  $\Tr_{4n} (\gamma) = \Tr_{4n} (\gamma^{-1}) = 1$ \\ 
\hline 
$l+3$ & $d_1 \mid (q^2-1), d_1 \nmid (q \pm 1)$ & $\Tr_{2n} (\gamma_1) = \Tr_{2n} (\gamma_1^{-1}) = 1$ \\ 
\hline 
$l+2$ & \vdots & $\Tr_{2n} (\gamma_2) = \Tr_{2n} (\gamma_2^{-1}) = 0$ \\
\hline
\vdots & \vdots & \vdots \\ 
\hline
$0$ & \vdots &  $\Tr_{2n} (\gamma_{l+4}) = \Tr_{2n} (\gamma_{l+4}^{-1}) = 0$ \\ 
\hline 
\end{tabular} 
\end{center}

\vspace{0.2cm}

In Case 1 we can have different sub-cases.
\begin{lemma}\label{lem_case_1}
In Case 1, if $l \geq 1$ and $\tilde{q} := \sqrt{q}$, then one of the following holds.
\begin{itemize}
\item \emph{Sub-case 1:} $d_2 \nmid (\tilde{q}+1)$. Then $d_i \nmid (\tilde{q} \pm 1)$ for all $i$ such that $2 \leq i \leq l+1$.
\item \emph{Sub-case 2:} $d_2 \mid (\tilde{q}+1)$. Then $d_i \mid (\tilde{q} - 1)$ for all $i \geq 3$.
\end{itemize}
\end{lemma}
\begin{proof}
We deal separately with the two sub-cases.

\begin{itemize}
\item \emph{Sub-case 1:} if $d_2 \nmid (\tilde{q}+1)$, then $d_2 \nmid (\tilde{q}-1)$ too. In fact, if $d_2 \mid (\tilde{q}-1)$, then $d_1 \mid (\tilde{q}+1)$ according to Lemma \ref{lem_lac}. This latter is not possible since $\gcd(q+1,\tilde{q}+1)=1$. Hence, if we define 
\begin{equation*}
q := \tilde{q}, \ n:= \frac{n}{2}, \ \gamma := \gamma_1, \ l:= l-1,\\
\end{equation*}
we have that either Case 2 or Case 3 holds for $\gamma$.
\item \emph{Sub-case 2:} if $d_2 \mid (\tilde{q}+1)$, then $d_3 \mid (\tilde{q}-1)$ according to Lemma \ref{lem_lac}. Therefore $d_i \mid (\tilde{q}-1)$ for all $i \geq 3$. 
\end{itemize}
\end{proof}

\begin{remark}
If Sub-case 2 holds in Lemma \ref{lem_case_1}, then we can define
\begin{equation*}
q := \tilde{q}, \ n:= \frac{n}{2}, \ \gamma := \gamma_1, \ l:= l-1.\\
\end{equation*}
Then Case 1 holds for $\gamma$.
\end{remark}

\begin{lemma}\label{lem_l3_2}
If $\alpha \in \F_{q^2}^{**}$ is an element of a connected component of $G_{q^2}$ such that $|\alpha|$ divides $q+1$, then $\alpha$ lies on the level $l+3$ or $2$ of $G_{q^2}$. Moreover, if $\alpha$ lies on the level $l+3$ (resp. $2$), then the order of any vertex at the level $l+3$ (resp. $2$) divides $q+1$.   
\end{lemma}
\begin{proof}
The first part of the claim follows from the fact that $\theta(\alpha)$ is a leaf of a connected component of $G_q$ and any tree in $G_q$ has depth $l+2$ or $1$. According to Lemma \ref{lem_lac} any leaf of $G_q$ can be expressed in the form $\theta (\gamma)$ for some $\gamma \in \F_{q^2}^{**}$ with $|\gamma| \mid (q+1)$. Since the depth is the same for any tree in a given component, we get the second part of the claim. 
\end{proof}

From Lemma \ref{lem_l3_2} we get the following.

\begin{corollary}\label{cor_l3_2}
Let $C_{q+1}$ and $C_{q^2+1}$ be  the cyclic subgroups of order $q+1$ and $q^2+1$ in $\F_{q^2}^*$ and $\F_{q^4}^*$ respectively. Then 
\begin{equation*}
C_{q+1} \subseteq \theta(C_{q^2+1}) \cup \theta^{l+2} (C_{q^2+1}).
\end{equation*}
\end{corollary}
\begin{proof}
Let $\alpha \in C_{q+1}$. Then $\alpha \in \F_{q^2}^*$, $\alpha \not \in \F_q$ and $\theta(\alpha) \in \F_q$. Since $\theta (\alpha)$ is a leaf of a connected component of $G_q$, we have that $\alpha$ lies on the level $2$ or $l+3$ of $G_{q^2}$. In the latter case $\alpha$ is a leaf of $G_{q^2}$ and $\alpha = \theta (\gamma)$ for some $\gamma \in C_{q^2+1}$. In the first case $\alpha$ belongs to a tree of $G_{q^2}$ having depth $l+3$, whose leaves have predecessors in $C_{q^2+1}$, namely $\alpha \in \theta^{l+2} (C_{q^2+1})$.
\end{proof}

\begin{lemma}
Let $\gamma$ be an element of $\F_{q^4}^{**}$ such that $|\gamma| \mid (q^2+1)$ but $|\theta (\gamma)| \nmid (q \pm 1)$. Let $p_1$ be the smallest prime which divides $q+1$ and $p_2$ the smallest prime which divides $q-1$. Then
\begin{equation}\label{ineq_1}
|\theta^i (\gamma)| \geq p_1 \cdot p_2 \geq (1 + 2^{l+1}) \cdot p_2 
\end{equation}
for any integer $i$ such that $1 \leq i \leq l+1$. Moreover, if $|\theta^{l+2} (\gamma)|$ does not divide $q+1$, the inequality (\ref{ineq_1}) holds also for $l+2 \leq i \leq l+4$.  
\end{lemma}
\begin{proof}
The first inequality in (\ref{ineq_1}) for the indices $i$ with $1 \leq i \leq l+1$ follows from the tables corresponding to Case 2 and 3 and from the fact that $\gcd(q+1,q-1) = 1$. If $|\theta^{l+2} (\gamma)|$ does not divide $q+1$, then the only  possible case is Case 3.

The second inequality in (\ref{ineq_1}) follows from Lemma \ref{elm_5}.
\end{proof}

From the tables above we deduce the following characterizations of some subsets of $A_n$ and $B_n$.
\begin{lemma}
Let 
\begin{align*}
A_{n_1} & := \{x \in A_n: \Tr_n (x) = \Tr_n (x^{-1}) = 1 \};\\
A_{n_0} & := \{x \in A_n: \Tr_n (x) = \Tr_n (x^{-1}) = 0 \};\\
B_{n_{01}} & := \{x \in B_n: \Tr_n (x) = 0, \ \Tr_n (x^{-1}) = 1 \};\\
B_{n_{10}} & := \{x \in B_n: \Tr_n (x) = 1, \ \Tr_n (x^{-1}) = 0 \}.
\end{align*}

Then 
\begin{align*}
A_{n_1} & := \theta^2 \{\gamma \in \F_{q^4}^{**}: |\gamma| \mid (q^2+1) \text{ and } |\theta(\gamma)| \mid (q+1) \};\\
A_{n_0} & := \bigcup_{i=3}^{l+4} \theta^i \{\gamma \in \F_{q^4}^{**}: |\gamma| \mid (q^2+1) \text{ and } |\theta(\gamma)| \mid (q+1) \};\\
B_{n_{01}} & := \theta^{l+3} \{\gamma \in \F_{q^4}^{**}: |\gamma| \mid (q^2+1) \text{ and } |\theta^{l+2} (\gamma)| \mid (q+1)  \};\\
B_{n_{10}} & := \theta^{l+4} \{\gamma \in \F_{q^4}^{**}: |\gamma| \mid (q^2+1) \text{ and } |\theta^{l+2} (\gamma)| \mid (q+1) \}.
\end{align*}
\end{lemma}

Finally we deduce the following.
\begin{lemma}
The map $\theta$ acts as a permutation on the set 
\begin{equation*}
\theta^{l+4} \{\gamma \in \F_{q^4}^{**}: |\gamma| \mid (q^2+1) \}.
\end{equation*}
\end{lemma}
\begin{proof}
If $\gamma \in \F_{q^4}^{**}$, then $\theta(\gamma)$ is a leaf of $G_{q^2}$ and belongs to the level $l+3$ of a tree in $G_{q^2}$. Hence $\theta^{l+4} (\gamma)$ is $\theta$-periodic.
\end{proof}

\subsection{Roots of Dickson polynomials}\label{Sub-dickson}
Elements of the form $x+x^{-1}$ play a crucial role in Dickson polynomials.

Let $q:=2^n$ for some positive integer $n$. We recall that the Dickson polynomial $D_m(x) \in \F_q [x]$ of the first kind of degree $m > 0$ with parameter $1$ is 
\begin{equation*}
D_m (x) := \sum_{i=0}^{\lfloor m/2 \rfloor} \frac{m}{m-i} \binom{m-i}{i} (-1)^i x^{m-2i}.
\end{equation*}
The following property holds:
\begin{displaymath}
D_m (x+x^{-1}) = x^m + x^{-m}.
\end{displaymath}
Equivalently, 
\begin{displaymath}
D_m (x+x^{-1}) = x^{-m} \cdot (x^m+1)^2.
\end{displaymath}
For any $\alpha \in \F_q$ we can find some $\gamma \in \F_{q^2}$ such that $\alpha = \gamma + \gamma^{-1}$. Therefore, finding a root $\alpha$ in $\F_q$ of $D_m(x)$ amounts to finding an element $\gamma \in \F_{q^2}$ such that
\begin{displaymath}
D_m (\gamma+{\gamma}^{-1}) = {\gamma}^{-m} \cdot ({\gamma}^m+1)^2 = 0,
\end{displaymath} 
which is equivalent to saying that $\gamma$ is an element of $\F_{q^2}^*$ such that $|\gamma|$ divides $m$.

Let $m$ be an integer such that $m>1$ and $m \mid (q+1)$.

We define the following sets:
\begin{align*}
S_m & := \{\alpha \in \F_q^*: D_m(\alpha) = D_m (\alpha^{-1}) = 0  \};\\
T_m & := \{\alpha \in \F_q^*: D_m (\alpha) = 0,  D_m (\alpha^{-1}) \not = 0 \}.
\end{align*}

In \cite{Blo} the authors investigated the existence of elements $\alpha \in \F_{q}^*$ belonging to $S_m$ or $T_m$. In \cite[Section 2]{Blo} they gave some existence and non-existence results. 
First they considered the case $m := q+1$, defining 
\begin{displaymath}
N_{q+1} := |S_{q+1}|.
\end{displaymath}

Then they proved the following result (see \cite[Theorem 2.1]{Blo}).

\begin{theorem}\label{thm_blo}
We have 
\begin{displaymath}
N_{q+1} = \dfrac{q+1+K(\chi_q)}{4},
\end{displaymath}
where $K(\chi_q) = \sum_{x \in \F_{q}^*} \chi_q (x+x^{-1})$ is a Kloosterman sum.
\end{theorem}

From Theorem \ref{thm_blo} and the fact that $|K(\chi_q)| \leq 2 q^{1/2}$ the authors of \cite{Blo} deduced the following.

\begin{corollary}
Let $n$ be a positive integer and $q = 2^n$. There is $\alpha \in \F_q^*$ such that both $\alpha$ and $\alpha^{-1}$ are roots of $D_{q+1} (x)$. When $q > 4$, there is $\alpha \in \F_q^*$ such that $\alpha$ is a root of $D_{q+1} (x)$, but $\alpha^{-1}$ is not. For $q =2, 4$, there is no $\alpha \in \F_q^*$ such that $\alpha$ is a root of $D_{q+1}(x)$, but $\alpha^{-1}$ is not.
\end{corollary}

Indeed, the fact that $S_{q+1} \not = \emptyset$ can be proved without computing $N_{q+1}$, just relying upon the Koblitz curve $E$ and on the graph $G_q$. We notice in passing that
\begin{equation*}
S_{q+1} = \{\theta(\gamma) : \gamma \in \F_{q^2}^{**}, |\gamma| \mid (q+1) \}.
\end{equation*} 
In the following we denote by $L_{A_n}$ and $L_{B_n}$ the sets of leaves of $G_q$ contained in $A_n$ and $B_n$ respectively.
\begin{lemma}\label{lem_l_a_n}
$L_{A_n}$ is not empty and $S_{q+1} = L_{A_n}$.
\end{lemma} 
\begin{proof}
We notice that $E(\F_2) \subseteq E(\F_{2^n})$ for all positive integers $n$. Since 
\begin{equation*}
E(\F_2) = \{O, (0,1), (1,1), (1,0) \},
\end{equation*}
we have that $|E(\F_{2^n})| \geq 4$. Therefore $A_n$ is not empty and the same holds for $L_{A_n}$. 

If $\alpha \in S_{q+1}$, then there exist two elements $\beta, \gamma \in \F_{q^2}^{**}$ such that $|\beta| \mid (q+1)$ and $|\gamma| \mid (q+1)$ with $\theta(\beta) = \alpha$ and $\theta(\gamma) = \alpha^{-1}$. In particular, $\alpha$ and $\alpha^{-1}$ are leaves of $G_q$ and are not $\theta$-periodic. Moreover $\theta(\alpha)$ is not $\theta$-periodic too. In fact, if $\alpha =1$, then $\theta(1) = 0$, which belongs to the tree rooted in $\infty$. If $\alpha \not = 1$, then $\alpha \not = \alpha^{-1}$. If $\theta (\alpha)$ were $\theta$-periodic, then the in-degree of $\theta(\alpha)$ in $G_q$ would be greater than $2$, which is absurd. Hence $\alpha$ and $\alpha^{-1}$ belong to a tree of depth at least $2$. We conclude that $\alpha, \alpha^{-1} \in L_{A_n}$. 

If $\alpha \in L_{A_n}$, then $\alpha^{-1} \in L_{A_n}$ too. According to Lemma \ref{lem_lac}, there exist two elements $\beta$ and $\gamma$ in $\F_{q^2}^{**}$ such that $|\beta| \mid (q+1)$ and $|\gamma| \mid (q+1)$ with $\theta(\beta) = \alpha$ and $\theta(\gamma) = \alpha^{-1}$. Hence $\alpha \in S_{q+1}$.
\end{proof}

In Lemma \ref{lem_l_b_n} we prove that $T_{q+1} \not = \emptyset$.

 We  recall that, for any $x \in \Pro(\F_q) \backslash \{0, \infty \}$, there are two distinct points in $E(\F_{q^2})$ having the same $x$-coordinate, while $(0,1)$ is the only point with $x=0$. In fact, if $P := (x_P, y_P) \in E(\F_{q^2}) \backslash \{ \infty \}$, then $-P = (x_P, x_P + y_P)$ (see \cite[Section 2]{SU2}). Moreover, $-P$ is the only point having the same $x$-coordinate of $P$ and $P = - P$ if and only if $x_P=0$.

\begin{lemma}\label{lem_l_b_n}
$L_{B_n}$ is not empty and $T_{q+1} = L_{B_n}$, provided that $n > 2$.
\end{lemma}
\begin{proof}
First we prove that $B_n \not = \emptyset$. Indeed,  
\begin{displaymath}
|A_n \cap \F_q^*| = \dfrac{|E(\F_q)|-2}{2}.
\end{displaymath} 
We recall that, according to Hasse bound,
\begin{displaymath}
|E(\F_q)| \leq q+1+2 \sqrt{q}.
\end{displaymath}
Therefore 
\begin{displaymath}
|A_n \cap \F_q^*| \leq \dfrac{q-1+2 \sqrt{q}}{2}.
\end{displaymath}
We have that 
\begin{displaymath}
\dfrac{q-1+2 \sqrt{q}}{2} < q-1 \Leftrightarrow 0 < q - 1 - 2 \sqrt{q}.
\end{displaymath}
This latter is true if $n > 2$. In fact, the real function $f(x) = x - 1 - 2 \sqrt{x}$  is strictly increasing for $x > 1$ and $f(8) > 0$.

Since $B_n \not = \emptyset$, also $L_{B_n} \not = \emptyset$.

If $\alpha \in T_{q+1}$, then there exists $\gamma \in \F_{q^2}^{**}$ such that $\theta(\gamma) = \alpha$ and $|\gamma| \mid (q+1)$ but there is no $\beta \in \F_{q^2}^{**}$ such that $\theta(\beta) = \alpha^{-1}$ and $|\beta| \mid (q+1)$. Therefore $\alpha^{-1} = \theta(\delta)$ for some $\delta \in \F_q^*$, according to Lemma \ref{lem_lac_0}. Hence $\Tr_n(\alpha) = 0$, while $\Tr_n(\alpha^{-1}) = 1$. Therefore $\alpha \in L_{B_n}$.

If $\alpha \in L_{B_n}$, then $\Tr_{n} (\alpha^{-1}) = 1$, while $\Tr_n(\alpha) = 0$. According to Lemma \ref{lem_lac} there exists $\beta \in \F_{q^2}^{**}$ such that $\theta(\beta) = \alpha$. Moreover, according to Lemma \ref{lem_lac_0}, there exists $\gamma \in \F_{q}^{*}$ such that $\theta (\gamma) = \alpha^{-1}$. Therefore $D_{q+1} (\alpha)  =0$, while $D_{q+1} (\alpha^{-1}) \not = 0$. Hence $\alpha \in T_{q+1}$ according to Lemma \ref{lem_lac_0}.
\end{proof}

\bibliography{Refs}
\end{document}